\documentclass[final]{siamltex}
\usepackage{graphicx}
\usepackage{color}
\usepackage{eufrak}      
\usepackage{dsfont}    
\usepackage{enumerate}

\newcommand{\EQ}{\begin{eqnarray}}
\newcommand{\EN}{\end{eqnarray}}
\newcommand{\EQQ}{\begin{eqnarray*}}
\newcommand{\ENN}{\end{eqnarray*}}
\newcommand{\nnum}{\nonumber}

\title{Linked Sections to ``On the Optimal Control of  Impulsive Hybrid Systems  On  Riemannian Manifolds"}

\author{Farzin Taringoo$^\ast$\hspace{-.5cm}\and Peter E. Caines\thanks{Department of Electrical and Computer Engineering and Centre for Intelligent
Machines, McGill University, Montreal, Canada, \{taringoo, peterc@cim.mcgill.ca\}.}}
\begin{document}
\maketitle
This paper contains material referenced in the content of the paper ``On the Optimal Control of  Impulsive Hybrid Systems  On  Riemannian Manifolds" submitted to the SIAM Journal on Control and Optimization on the 20th September 2012 and which is also to be found on math arxiv http://arxiv.org/abs/1209.4609. The contents of this manuscript consists of :
\begin{enumerate}
\item Interior Optimal Switching States,
\item Time Varying Switching Manifolds and Discontinuity of the Hamiltonian,
\item Time Varying Impulsive Jumps,
\item Interior Optimal Switching States, Time Varying Switching Manifolds and Impulsive Jumps,
\item Extension to Multiple Switchings Cases,
\end{enumerate}
where  Items  1 - 5 correspond to the referencing  link on page 2 of SIAM
Journal on Control and Optimization, Submission 086781.
\section{Interior Optimal Switching States}

Here we specify a hypothesis for MHOCP which expresses the HMP statement based on a differential form of the hybrid value function.\\\\
\textbf{\textit{A4}}: For an MHOCP, the value function $v(x,t),\hspace{.2cm} x\in \mathcal{M},t\in (t_{0},t_{f})$, is assumed to be differentiable at the optimal switching state $x^{o}(t^{-}_{s})$ in the switching manifold $\mathcal{S},$ where the optimal switching state is an interior point of the attainable switching states on the switching manifold.\\\\
We note that $\textbf{\textit{A4}}$ rules out MHOCPs derived from BHOCPs (see Lemma 3.2).
The following theorem gives the HMP statement for an accessible  MHOCP satisfying \textbf{\textit{A4}}. 
\begin{theorem}
\label{t20}
Consider an impulsive MHOCP satisfying \textbf{\textit{A1-A4}}. Then  corresponding to  the  optimal  control and optimal state trajectory $u^{o},x^{o}$ with a single switching state at $(x^{o}(t_{s}),t_{s})$, there exists a nontrivial adjoint trajectory $\lambda^{o}(.)=(x^{o}(.),p^{o}(.))\in T^{*}\mathcal{M}$ defined along the optimal state trajectory such that:
\EQ \hspace{.2cm}H_{q_{i}}(x^{o}(t),p^{o}(t),u^{o}(t))\leq H_{q_{i}}(x^{o}(t),p^{o}(t),u_{1}), \forall u_{1}\in U, t\in[t_{0},t_{f}],i=0,1, \EN
and the corresponding optimal adjoint variable $\lambda^{o}(.)\in T^{*}\mathcal{M}$ satisfies:
\EQ \dot{\lambda^{o}}(t)=\overrightarrow{H}_{q_{i}}(\lambda^{o}(t)),\quad t\in [t_{0},t_{f}], i=0,1.\EN
At the optimal switching state $x^{o}({t_{s}})$ and switching time $t_{s}$, we have 

\EQ\label{ad} &&\hspace{-.5cm}p^{o}(t^{-}_{s})=T^{*}\zeta(p^{o}(t_{s}))+\mu dv(x^{o}(t^{-}_{s}),t_{s}),\nnum\\&& \hspace{-.5cm} p^{o}(t^{-}_{s})\in T^{*}_{x^{o}(t^{-}_{s})}\mathcal{M},\quad p^{o}(t_{s})\in T^{*}_{x^{o}(t_{s})}\mathcal{M},\nnum\\&&\hspace{-.5cm} x^{o}(t_{s})=\zeta(x^{o}(t^{-}_{s})),\EN
where $\mu\in \mathds{R}, T^{*}\zeta:T^{*}\mathcal{M}\rightarrow T^{*}\mathcal{M},$
 and 
 \EQ dv(x^{o}(t^{-}_{s}),t_{s})=\sum^{n}_{j=1}\frac{\partial v(x^{o}(t^{-}_{s}),t_{s})}{\partial x^{j}}dx^{j}\in T^{*}_{x^{o}(t_{s})}\mathcal{M}. \EN
The continuity of the Hamiltonian at $(x^{o}({t_{s}}),t_{s})$ is given as follows
\EQ\hspace{-.5cm} H_{q_{0}}(x^{o}(t^{-}_{s}),p^{o}(t^{-}_{s}),u^{o}(t^{-}_{s}))=H_{q_{1}}(x^{o}(t_{s}),p^{o}(t_{s}),u^{o}(t_{s})).\EN
\end{theorem}
\begin{proof}
The proof closely parallels the proof of Theorem 4.5 with the role of $dN_{x(t)}$ being replaced by $dv(x,t)$ whose existence is guaranteed by \textbf{\textit{A4}}; this is presented in Appendix \ref{s2}.
\end{proof}


\section{Time Varying Switching Manifolds and Discontinuity of the Hamiltonian}
\label{sec4}
In this section we extend the results obtained in the previous section to impulsive autonomous hybrid systems with  time varying switching manifolds. The HMP proof  parallels the proof of time invariant cases with a  modification in the  variation of the value function $v(x,t)$ with respect to the switching time. 
Since $\mathcal{S}$ is time varying,  we decompose the metric of $\mathcal{M}\times \mathds{R}$ as
\EQ g_{\mathcal{M}\times \mathds{R}}=g_{\mathcal{M}}\oplus g_{R},\EN
where $g_{R}$ is the Euclidean metric of $\mathds{R}$.
Now the one form corresponding to the normal vector $N_{(x,t)}$ at $(x,t)\in\mathcal{S}\subset \mathcal{M}\times \mathds{R}$ is defined as 
\EQ\label{kir11} dN_{(x,t)}:=g_{\mathcal{M}\times \mathds{R}}(N_{(x,t)},.)\in T^{*}_{(x,t)}(\mathcal{M}\times \mathds{R})=T^{*}_{x}\mathcal{M}\oplus T^{*}_{t}\mathds{R}.\EN
Based on the special form of $g_{\mathcal{M}\times \mathds{R}}$, we can decompose $dN_{(x,t)}$ as
\EQ dN_{(x,t)}=dN_{x}\oplus dN_{t},\hspace{.2cm} dN_{x}\in T^{*}_{x}\mathcal{M}, dN_{t}\in T^{*}_{t}\mathds{R}\simeq\mathds{R}.\EN

\begin{theorem}
\label{tt2}
Consider an impulsive MHOCP satisfying hypotheses \textbf{\textit{A1-A3}} where the switching manifold is an $n$ dimensional embedded time varying switching submanifold $\mathcal{S}\subset\mathcal{M}\times \mathds{R}$ and where the switching state jump is given by a smooth function $\zeta:\mathcal{M}\rightarrow \mathcal{M}$ whenever $(x(t^{-}),t)\in \mathcal{S}$. Then  corresponding to  the  optimal  control and optimal trajectory $u^{o},x^{o}$ with a single switching state at $(x^{o}(t_{s}),t_{s})$, there exists a nontrivial adjoint trajectory  $\lambda^{o}(.)=(x^{o}(.), p^{o}(.))\in T^{*}\mathcal{M}$ defined along the optimal state trajectory such that:
\EQ &&\hspace{-.5cm}H_{q_{i}}(x^{o}(t),p^{o}(t),u^{o}(t))\leq H_{q_{i}}(x^{o}(t),p^{o}(t),u_{1}),\quad \forall u_{1}\in U, t\in[t_{0},t_{f}],i=0,1,\nnum\\ \EN
and the corresponding optimal adjoint variable $\lambda^{o}(.)\in T^{*}\mathcal{M}$, (locally given by $\lambda^{o}(.)=(x^{o}(.),p^{o}(.))$) satisfies:
\EQ \dot{\lambda^{o}}(t)=\overrightarrow{H}_{q_{i}}(\lambda^{o}(t)),\quad t\in [t_{0},t_{f}], i=0,1.\EN
At the optimal switching state $x^{o}({t_{s}})$ and switching time $t_{s}$, there exists $dN_{x}\in T^{*}_{x}\mathcal{M}$ such that 

\EQ &&\hspace{-.5cm}p^{o}(t^{-}_{s})=T^{*}\zeta(p^{o}(t_{s}))+\mu dN_{x^{o}(t^{-}_{s})},\nnum\\&& \hspace{-.5cm} p^{o}(t^{-}_{s})\in T^{*}_{x^{o}(t^{-}_{s})}\mathcal{M},\quad p^{o}(t_{s})\in T^{*}_{x^{o}(t_{s})}\mathcal{M},\nnum\\&&\hspace{-.5cm} x^{o}(t_{s})=\zeta(x^{o}(t^{-}_{s})),\EN
\EQ x(0)=x^{o}_{0},\hspace{.2cm} p^{o}(t_{f})=dh(x^{o}(t_{f}))\in T^{*}_{x^{o}(t_{f})}\mathcal{M},\hspace{.2cm} dh= \sum^{n}_{i=1}\frac{\partial h}{\partial x^{i}}dx^{i}\in T^{*}_{x}\mathcal{M},\EN
where $\mu\in \mathds{R}$ and $T^{*}\zeta:T^{*}\mathcal{M}\rightarrow T^{*}\mathcal{M}$.
The discontinuity of the Hamiltonian at $(x^{o}({t_{s}}),t_{s})$ is given by
\EQ H_{q_{0}}(x^{o}(t^{-}_{s}),p^{o}(t^{-}_{s}),u^{o}(t^{-}_{s}))=H_{q_{1}}(x^{o}(t_{s}),p^{o}(t_{s}),u^{o}(t_{s}))-\mu \langle dN_{t_{s}},\frac{\partial}{\partial t}\rangle,\nnum\\\EN
where $dN_{t_{s}}$ is the differential form corresponding to the time component of the normal vector at $(x({t^{-}_{s}}),t_{s})$ on the switching manifold $\mathcal{S}$. \\
\end{theorem}
 \begin{proof}
The proof is given in Appendix \ref{s5}.
\end{proof}
\section{Time Varying Impulsive Jumps}
In this section we investigate the HMP equations in the case of time varying impulsive jumps. For a HOCP with two discrete states, consider the state jump function as a smooth time varying map $\hat{\zeta}:\mathcal{M}\times \mathds{R}\rightarrow \mathcal{M}$. Therefore $T\hat{\zeta}:T\mathcal{M}\oplus T\mathds{R}\rightarrow T\mathcal{M}$ and $T^{*}\hat{\zeta}:T^{*}\mathcal{M}\rightarrow  T^{*}\mathcal{M}\oplus T\mathds{R} $.
we denote $T\hat{\zeta}=T\zeta\oplus D_{t}\zeta$, where
\EQ T\zeta:T\mathcal{M}\rightarrow T\mathcal{M},\quad D_{t}\zeta:T\mathds{R}\rightarrow T\mathcal{M},\EN
where $T\zeta$ and $D_{t}\zeta$ are the pushforwards of $\hat{\zeta}$ with respect to  $t\in \mathds{R}$ and $x\in \mathcal{M}$ respectively. The following theorem gives the HMP for hybrid impulsive systems in the case of time varying impulse jumps which is consistent with the results presented in \cite{Reidinger2}.
\begin{theorem}
\label{tt22}
Consider an impulsive MHOCP satisfying hypotheses \textbf{\textit{A1-A3}}. The switching manifold is assumed to be an $n$ dimensional embedded time varying submanifold $\mathcal{S}\subset\mathcal{M}\times \mathds{R}$ and the switching state jump is given by a time varying smooth function $\hat{\zeta}:\mathcal{M}\times \mathds{R}\rightarrow \mathcal{M}$ which is enabled whenever $(x(t^{-}),t)\in \mathcal{S}$; then  corresponding to  the  optimal  control and optimal trajectory $u^{o},x^{o}$, with a single switching state at $(x^{o}(t_{s}),t_{s})$,  there exists a nontrivial adjoint trajectory $\lambda^{o}(.)=(x^{o}(.),p^{o}(.))\in T^{*}\mathcal{M}$ defined along the optimal state trajectory such that:
\EQ &&\hspace{-.5cm}H_{q_{i}}(x^{o}(t),p^{o}(t),u^{o}(t))\leq H_{q_{i}}(x^{o}(t),p^{o}(t),u_{1}),\quad \forall u_{1}\in U, t\in[t_{0},t_{f}],i=0,1,\nnum\\ \EN
and the corresponding optimal adjoint trajectory $\lambda^{o}(.)\in T^{*}\mathcal{M}$, locally given by $\lambda^{o}(.)=(x^{o}(.),p^{o}(.))$, satisfies
\EQ \dot{\lambda^{o}}(t)=\overrightarrow{H}_{q_{i}}(\lambda^{o}(t)),\quad t\in [t_{0},t_{f}], i=0,1.\EN
At the optimal switching state $x^{o}({t_{s}})$ and switching time $t_{s}$, there exists $dN_{x}\in T^{*}_{x}\mathcal{M}$ such that 

\EQ &&\hspace{-.5cm}p^{o}(t^{-}_{s})=T^{*}\zeta(p^{o}(t_{s}))+\mu dN_{x^{o}(t^{-}_{s})},\nnum\\&& \hspace{-.5cm} p^{o}(t^{-}_{s})\in T^{*}_{x^{o}(t^{-}_{s})}\mathcal{M},\quad p^{o}(t_{s})\in T^{*}_{x^{o}(t_{s})}\mathcal{M},\nnum\\&&\hspace{-.5cm} x^{o}(t_{s})=\zeta(x^{o}(t^{-}_{s})),\EN
\EQ \hspace{.3cm}x^{o}(t_{0})=x_{0},\hspace{.2cm} p^{o}(t_{f})=dh(x^{o}(t_{f}))\in T^{*}_{x^{o}(t_{f})}\mathcal{M},\hspace{.2cm} dh= \sum^{n}_{i=1}\frac{\partial h}{\partial x^{i}}dx^{i}\in T^{*}_{x}\mathcal{M},\EN
where $\mu\in \mathds{R}$, 
\EQ T^{*}\hat{\zeta}=T^{*}\zeta\oplus D^{*}_{t}\zeta:T^{*}\mathcal{M}\rightarrow T^{*}\mathcal{M}\oplus T^{*}\mathds{R},\EN
and
\EQ T^{*}\zeta:T^{*}\mathcal{M}\rightarrow T^{*}\mathcal{M},\quad D^{*}_{t}\zeta:T^{*}\mathcal{M}\rightarrow T^{*}\mathds{R}.\EN
The discontinuity of the Hamiltonian at $(x^{o}({t_{s}}),t_{s})$ is given by
\EQ &&H_{q_{0}}(x^{o}(t^{-}_{s}),p^{o}(t^{-}_{s}),u^{o}(t^{-}_{s}))=\nnum\\&&H_{q_{1}}(x^{o}(t_{s}),p^{o}(t_{s}),u^{o}(t_{s}))-D^{*}_{t}\zeta(p^{o}(t_{s}))-\mu \langle dN_{t_{s}},\frac{\partial}{\partial t}\rangle.\EN
\end{theorem}
\begin{proof}
The proof is given in Appendix \ref{s6}.
\end{proof}
\section{Interior Optimal Switching States, Time Varying Switching Manifolds and Impulsive Jumps}
In this section we extend Theorem \ref{tt2} to MHOCPs satisfying \textbf{\textit{A4}} where the switching manifold $\mathcal{S}$ and the impulsive jump $\hat{\zeta}$ are both time varying. The results here are consistent with the results presented in \cite{Reidinger2}.
\\\\
In the case where the switching manifold is a time variant submanifold $\mathcal{S}\subset \mathcal{M}\times \mathds{R}$, we have  
\EQ \hat{d}v(x,t)\in T^{*}_{(x,t)}(\mathcal{M}\times \mathds{R})=T^{*}_{x}\mathcal{M}\oplus T^{*}_{t}\mathds{R},\EN
where locally
\EQ  &&\hspace{-.7cm}\hat{dv}(x^{o}(t^{-}_{s}),t_{s})=\sum^{n}_{j=1}\frac{\partial v(x^{o}(t^{-}_{s}),t_{s})}{\partial x^{j}}dx^{j}+D^{*}_{t}v(x^{o}(t^{-}_{s}),t_{s})dt\in T^{*}_{x^{o}(t^{-}_{s})}\mathcal{M}\oplus T^{*}_{t_{s}}\mathds{R}.\nnum\\\EN
The following lemma is an extension of Lemma \ref{l5} on time varying switching manifolds.
\begin{lemma}
\label{ll5}
For an MHOCP with a single switching from the discrete state $q_{0}$ to the discrete state $q_{1}$ at the unique switching time $t_{s}$ on the optimal trajectory $(x^{o}(.),u^{o}(.))$ and  an embedded time varying switching manifold $\mathcal{S}\subset \mathcal{M}\times \mathds{R}$ of dimension $k\leq dim(\mathcal{M})$; then at the optimal switching state and time $(x^{o}(t^{-}_{s}),t_{s})\in \mathcal{S}$, 
\EQ \langle \hat{d}v(x^{o}(t^{-}_{s}),t_{s}),X\rangle=0, \quad \forall X\in T_{(x^{o}(t^{-}_{s}),t_{s})}\mathcal{S}.\EN
\end{lemma}
\begin{proof}
The proof parallels the proof of Lemma \ref{l5} concerning the extra variable $t_{s}$.
\end{proof}
\begin{theorem}
\label{tt222}
Consider an impulsive MHOCP satisfying hypotheses \textbf{\textit{A1-A4}}; then  corresponding to  the  optimal  control and optimal trajectory $u^{o},x^{o}$,  there exists a nontrivial adjoint trajectory $\lambda^{o}(.)=(x^{o}(.),p^{o}(.))\in T^{*}\mathcal{M}$ defined along the optimal state trajectory such that:
\EQ &&\hspace{-.5cm}H_{q_{i}}(x^{o}(t),p^{o}(t),u^{o}(t))\leq H_{q_{i}}(x^{o}(t),p^{o}(t),u_{1}),\quad \forall u_{1}\in U, t\in[t_{0},t_{f}],i=0,1,\nnum\\ \EN
and the corresponding optimal adjoint variable $\lambda^{o}(.)\in T^{*}\mathcal{M}$, locally given as $\lambda^{o}(.)=(x^{o}(.),p^{o}(.))$, satisfies
\EQ \dot{\lambda^{o}}(t)=\overrightarrow{H}_{q_{i}}(\lambda^{o}(t)),\quad t\in [t_{0},t_{f}], i=0,1.\EN
At the optimal switching state $x^{o}({t_{s}})$ and switching time $t_{s}$, we have 

\EQ &&\hspace{-.5cm}p^{o}(t^{-}_{s})=T^{*}\zeta(p^{o}(t_{s}))+\mu dv(x^{o}(t^{-}_{s}),t_{s}),\nnum\\&& \hspace{-.5cm} p^{o}(t^{-}_{s})\in T^{*}_{x^{o}(t^{-}_{s})}\mathcal{M},\quad p^{o}(t_{s})\in T^{*}_{x^{o}(t_{s})}\mathcal{M},\nnum\\&&\hspace{-.5cm} x^{o}(t_{s})=\zeta(x^{o}(t^{-}_{s})),\EN
\EQ \hspace{.3cm}x^{o}(t_{0})=x_{0},\hspace{.2cm} p^{o}(t_{f})=dh(x^{o}(t_{f}))\in T^{*}_{x^{o}(t_{f})}\mathcal{M},\hspace{.2cm} dh= \sum^{n}_{i=1}\frac{\partial h}{\partial x^{i}}dx^{i}\in T^{*}_{x}\mathcal{M},\EN
where $\mu\in \mathds{R}$, 
\EQ T^{*}\hat{\zeta}=T^{*}\zeta\oplus D^{*}_{t}\zeta:T^{*}\mathcal{M}\rightarrow T^{*}\mathcal{M}\oplus T^{*}\mathds{R},\EN
and
\EQ T^{*}\zeta:T^{*}\mathcal{M}\rightarrow T^{*}\mathcal{M},\quad D^{*}_{t}\zeta:T^{*}\mathcal{M}\rightarrow T^{*}\mathds{R}.\EN
The discontinuity of the Hamiltonian at $(x^{o}(t^{-}_{s}), t_{s}),$  is given as follows:
\EQ &&H_{q_{0}}(x^{o}(t^{-}_{s}),p^{o}(t^{-}_{s}),u^{o}(t^{-}_{s}))=\nnum\\&&H_{q_{1}}(x^{o}(t_{s}),p^{o}(t_{s}),u^{o}(t_{s}))-D^{*}_{t}\zeta(p^{o}(t_{s}))-\mu D^{*}_{t}v(x^{o}(t^{-}_{s}),t_{s}).\EN

\end{theorem}
\begin{proof}
The proof parallels that of Theorem \ref{tt2} and employs the results of Lemma \ref{ll5}.
\end{proof}

\section{Extension to Multiple Switchings Cases}
\label{sec5}
In this section we obtain the HMP theorem statement for multiple switching hybrid systems where switching manifolds are time invariant. The standing assumption in this section is that $x^{o}(.)$ is an optimal trajectory under the optimal control  $u^{0}(.)$ for a given MHOCP; it is further assumed that this is a sequence of autonomous transitions along $x^{o}(.)$ at the distinct time instants $t_{0},t_{1},...,t_{L}$ and $\mathcal{S}_{i}$ is a time invariant switching manifold subcomponent of $\mathcal{M}$. 
\begin{lemma}
\label{l77}
Without loss of generality, assume that for all sufficiently small $0\leq \epsilon$ the needle variation $u_{\pi}(t,.)$  applied at  $t^{1}$, $t_{j-1}<t^{1}<t_{j}$, the resulting perturbed trajectories intersect  only $\mathcal{S}_{i}:=n_{q_{i}, q_{i+1}}, i=0,...,L$ and assume further that switching times are greater than the optimal switching times, i.e. $t_{i}\leq t_{i}(\epsilon),\quad i=j,...,L$. Then the state variation at $t_{f}$ is given as 
\EQ \label{122}&&\hspace{-.7cm}\frac{d}{d\epsilon}\Phi_{\pi}^{(t_{f},t^{1}),x}|_{\epsilon=0}=\Big(\prod^{L-j}_{i=0}  T\Phi^{(t_{i+j+1},t_{i+j})}_{f_{q_{i+j}}}\circ T\zeta_{i+j+1}\Big)\circ T\Phi^{(t_{j},t^{1})}_{f_{q_{i}}} \nnum\\&&\hspace{-.7cm}\times\big(f_{q_{i}}(x^{o}(t^{1}),u_{1})-f_{q_{i}}(x^{o}(t^{1}),u^{o}(t^{1}))\big)+\sum^{L-j}_{i=0}\big(\prod^{L-j}_{l=i}T\Phi^{(t_{l+j+1},t_{l+j})}_{f_{q_{l+j}}}\circ T\zeta_{l+j+1}\big)\nnum\\&&\hspace{-.7cm}\times\Big(\frac{d t_{i+j}(\epsilon)}{d\epsilon}|_{\epsilon=0}\big(f_{q_{i+j+1}}(x^{o}(t_{i+j}), u^{o}(t_{i+j}))-T\zeta_{i+j}f_{q_{i+j}}(x^{o}(t^{-}_{i+j}), u^{o}(t^{-}_{i+j}))\big)\Big)\nnum\\&&\hspace{2.7cm} \in T_{x(t_{f})}\mathcal{M},\EN
where $T\zeta_{L+1}=I$ and for simplicity  we use $\zeta_{i}$ instead of $\zeta_{q_{i},q_{i+1}}$ for $ i=0,...,L$.
\end{lemma}
\begin{proof}
The proof is based on the results of Lemma 3.1 and an extension of (C.6) and (C.7) to the case where $t_{i+j}(\epsilon)$ is the $(i+j)$th switching time corresponding to $u_{\pi}(t,\epsilon)$.  
\end{proof}

Employing the Lemma above, Lemma 4.4 can be generalized to multiple switching hybrid systems as follows:
\begin{lemma}
\label{l66}
For a HOCP corresponding to a given sequence of event transitions $ i=0,...,L,$  we have
\EQ \langle dh(x^{o}(t_{f})),v_{\pi}(t_{f})\rangle\geq 0,\quad \forall v_{\pi}(t_{f})\in K_{t_{f}},\EN
where
\EQ K_{t_{f}}=\bigcup^{L}_{r=1}K^{r}_{t_{f}},\EN
and
\EQ \label{222}&&\hspace{0cm}K^{r}_{t_{f}}=\bigcup_{t_{r-1}\leq t< t_{r}}\bigcup_{u_{1}\in U}\Big(\prod^{L-r}_{i=0}  T\Phi^{(t_{i+r+1},t_{i+r})}_{f_{q_{i+r}}}\circ T\zeta_{i+r+1}\Big)\circ\big\{ T\Phi^{(t_{r},t)}_{f_{q_{r}}}\big(f_{q_{r}}(x^{o}(t),u_{1})\nnum\\&&\hspace{0cm}-f_{q_{r}}(x^{o}(t),u^{o}(t))\big)\big\}+\bigcup_{t_{r-1}\leq t< t_{r}}\bigcup_{u_{1}\in U}\sum^{L-r}_{i=0}\big(\prod^{L-r}_{l=i}T\Phi^{(t_{l+r+1},t_{l+r})}_{f_{q_{l+r}}}\circ T\zeta_{l+j+1}\big)(\frac{dt_{i+r}(\epsilon)}{d\epsilon}|_{\epsilon=0}\nnum\\&&\times\Big(f_{q_{i+r+1}}(x^{o}(t_{i+r}), u^{o}(t_{i+r}))-T\zeta_{i+r}f_{q_{i+r}}(x^{o}(t^{-}_{i+r}), u^{o}(t^{-}_{i+r}))\Big),\nnum\\\EN
\end{lemma}
\begin{proof}
The proof parallels the proof of Lemma 4.4 and employs the results of Lemma \ref{l77}.
\end{proof}

The following theorem gives the HMP statement for the case of multiple switchings impulsive hybrid systems.
\begin{theorem}
\label{t4}
Consider a multiple switching  impulsive MHOCP satisfying hypotheses  \textbf{\textit{A1-A3}}. Then  corresponding to  the  optimal  control and optimal state trajectory $u^{o},x^{o}$,  there exists a nontrivial $\lambda^{o}(.)\in T^{*}\mathcal{M}$ along the optimal state trajectory such that:
\EQ \hspace{-.3cm}H_{q_{i}}(x^{o}(t),p^{o}(t),u^{o}(t))\leq H_{q_{i}}(x^{o}(t),p^{o}(t),u_{1}), \hspace{.2cm}\forall u_{1}\in U, t\in[t_{0},t_{f}], i=0,...,L,\nnum\\\EN
and the corresponding optimal adjoint trajectory $\lambda^{o}(.)\in T^{*}\mathcal{M}$,  locally given by $\lambda^{o}(.)=(x^{o}(.),p^{o}(.))$, satisfies:
\EQ \dot{\lambda^{o}}(t)=\overrightarrow{H}_{q_{i}}(\lambda^{o}(t)),\quad t\in [t_{0},t_{f}],\hspace{.1cm} i=0,...,L.\EN
 At the optimal switching state $x^{o}({t_{i}})$ and switching time $t_{i},$ there exists $dN^{i}_{x}\in T^{*}_{x}\mathcal{S}_{i}$ such that

\EQ &&\hspace{-.5cm}p^{o}(t^{-}_{i})=T^{*}\zeta_{i}(p^{o}(t_{i}))+\mu_{i} dN^{i}_{x^{o}(t^{-}_{i})},\nnum\\&& \hspace{-.5cm} p^{o}(t^{-}_{i})\in T^{*}_{x^{o}(t^{-}_{i})}\mathcal{M},\quad p^{o}(t_{i})\in T^{*}_{x^{o}(t_{i})}\mathcal{M},\nnum\\&&\hspace{-.5cm} x^{o}(t_{i})=\zeta_{i}(x^{o}(t^{-}_{i})),\EN
where $\mu_{i}\in \mathds{R}$ and $T^{*}\zeta_{i}:T^{*}\mathcal{M}\rightarrow T^{*}\mathcal{M}.$
The continuity of the Hamiltonian at the switching instants $(x^{o}(t^{-}_{i}), t_{i}), i=0,...,L,$ is given by
\EQ\hspace{-.5cm} H_{q_{i}}(x^{o}(t^{-}_{i}),p^{o}(t^{-}_{i}),u^{o}(t^{-}_{i}))=H_{q_{i+1}}(x^{o}(t_{i}),p^{o}(t_{i}),u^{o}(t_{i})),\hspace{.2cm} i=0,...,L.\EN

\end{theorem}
\begin{proof}
The proof parallels the proof of Theorem 4.5 employing the results of Lemma \ref{l66}.
\end{proof}

\appendix
 \section{Proof of Theorem \ref{t20}}
\label{s2}
The following results for the variation of the hybrid value function is presented and then a complete proof of  Theorem \ref{t20} is provided.

Since  $\mathcal{S}$ is an embedded submanifold of $\mathcal{M}$, necessarily there  exists an embedding inclusion $i$ from $\mathcal{S}$ to $i(\mathcal{S})\subset \mathcal{M} $. The push-forward of  $i$ is given as
\EQ Ti:T_{x}\mathcal{S}\rightarrow T_{x}\mathcal{M}.\EN
For any tangent vector $X\in T_{x}\mathcal{S}$, the image vector $Ti(X)\in T_{x}\mathcal{M}$ is a tangent vector on $\mathcal{M}$.
There exists a local coordinate representation of $X$, i.e. $X=\sum^{n}_{j=1}X^{j}\frac{\partial}{\partial x^{j}}$, such that   $X\in T_{x}\mathcal{S}$ if and only if $X^{j}=0,\quad  j>k,$ where $k$ is the dimension of $\mathcal{S}$, see \cite{Lee2}. The following lemma gives a relation between  \\$dv(x^{o}(t^{-}_{s}),t_{s})=\sum^{n}_{j=1}\frac{\partial v(x^{o}(t_{s}),t_{s})}{\partial x^{j}}dx^{j}\in T^{*}_{x^{o}(t^{-}_{s})}\mathcal{M}$  where $v(x^{o}(t_{s}),t_{s})$ is smooth by \textbf{\textit{A4}} and a tangent vector $X\in T_{x^{o}(t^{-}_{s})}\mathcal{M}$ which is also a tangent vector in $T_{x^{o}(t^{-}_{s})}\mathcal{S}$ in the local coordinate system given above. The statement of the following lemma is given for a general embedded submanifold  $\mathcal{S}$ which is not necessarily $n-1$ dimensional.
\begin{lemma}
\label{l5}
Consider  an  MHOCP with a single switching from the discrete state $q_{0}$ to the discrete state $q_{1}$ at the unique switching time $t_{s}$ on the optimal trajectory $(x^{o}(.),u^{o}(.))$ and a $k$ dimensional  embedded switching manifold $\mathcal{S}\subset \mathcal{M}$ satisfying \textbf{\textit{A1}}-\textbf{\textit{A4}}; then at the optimal switching state $x^{o}(t_{s})\in \mathcal{S}$  and switching time $t_{s}$, we have 
\EQ \langle dv(x^{o}(t^{-}_{s}),t_{s}),X\rangle=0, \quad \forall X\in Ti(T_{x^{o}(t^{-}_{s})}\mathcal{S}).\EN
\end{lemma}
\begin{proof}
Since $X\in Ti(T_{x^{o}(t^{-}_{s})}\mathcal{S})$ there exists $X_{\mathcal{S}}\in T_{x^{o}(t^{-}_{s})}\mathcal{S}$ such that $X=Ti(X_{\mathcal{S}})$. By applying the same extension method employed in Lemma 4.4, we extend $X_{\mathcal{S}}$ to a vector field $X^{'}_{\mathcal{S}}\in \mathfrak{X}(\mathcal{S})$ such that $X^{'}_{\mathcal{S}}(x^{o}(t^{-}_{s}))=X_{\mathcal{S}}$.

Let us denote the induced Riemannian metric from $\mathcal{M}$ to $\mathcal{S}$ as $g_{\mathcal{S}}$.  By the fundamental theorem of  existence of geodesics and the Taylor expansion on Riemannian manifolds we have 
\EQ \label{tay2}v((exp_{x^{o}(t^{-}_{s})}\theta X_{\mathcal{S}}),t_{s})&=&v(x^{o}(t^{-}_{s}),t_{s})+\theta (\nabla^{'}_{X^{'}_{\mathcal{S}}}v)(x^{o}(t^{-}_{s}),t_{s})+o(\theta),\nnum\\&& 0<\theta<\theta^{*},\EN
where $\nabla^{'}$ is the Levi-Civita connection of $\mathcal{S}$ with respect to the induced metric $g_{\mathcal{S}}$. Since $\mathcal{S}$ is an embedded submanifold of $\mathcal{M}$, the inclusion map is a full rank homeomorphism from $\mathcal{S}$ to $i(\mathcal{S})$, therefore, for each $X\in  Ti(T_{x^{o}(t^{-}_{s})}\mathcal{S})$, the corresponding $X_{\mathcal{S}}$ is unique. The vector space property of $T_{x^{o}(t^{-}_{s})}\mathcal{S}$ implies $-X_{\mathcal{S}}\in T_{x^{o}(t_{s})}\mathcal{S}$, hence, by the optimality of $x^{o}(t^{-}_{s})$ on $\mathcal{S}$ and the accessibility of $\dot{x}(t)=f_{q_{0}}(x,u)$, an application of  (\ref{tay2}) to $v$ along $-X_{\mathcal{S}}$ gives
\EQ\label{ll1}  \nabla^{'}_{X^{'}_{\mathcal{S}}}v=0,\quad \forall X^{'}_{\mathcal{S}}\in T_{x^{o}(t^{-}_{s})}\mathcal{S}.\EN
(\ref{ll1}) and (4.4)(ii)  together imply 
\EQ \label{ll2}\frac{\partial v}{\partial x^{j}}(x^{o}(t^{-}_{s}),t_{s})=0,\quad j=1,...,k,\EN
where $k$ is the dimension of $\mathcal{S}$.
In the local coordinates of $x^{o}(t_{s})\in \mathcal{M}$, (\ref{ll2}) yields  
\EQ\langle dv(x^{o}(t^{-}_{s}),t_{s}),X\rangle=\langle \sum^{n}_{j=1}\frac{\partial v(x^{o}(t^{-}_{s}),t_{s})}{\partial x^{j}}dx^{j}, \sum^{n}_{j=1}X^{j}\frac{\partial}{\partial x^{j}}\rangle, \EN
where (\ref{ll2})  together with $X^{j}=0,  j>k$ completes the proof.
\end{proof}
\\\\
The proof of Theorem \ref{t20} is then given as follows:
\begin{proof}
The proof parallels the proof of Theorem 4.5 where by Lemma \ref{l5}, $dN_{x^{o}(t^{-}_{s})}$ is replaced by $dv(x^{o}(t^{-}_{s}),t_{s})$.\end{proof}\\

\section{Proof of Theorem \ref{tt2}}
\label{s5}
\begin{proof}
The first step of the proof of Theorem 4.5 is unchanged. For the control needle  variation before the optimal switching time $t_{s}$, i.e. step 2, in case $(i)$: $t_{s}\leq  t_{s}(\epsilon)$, we have 
\EQ\label{llq} \frac{d \Phi_{\pi, f_{q_{0}}}^{(t^{-}_{s}(\epsilon),t^{1}),x} }{d\epsilon}|_{\epsilon=0}\oplus\frac{d t_{s}(\epsilon)}{d \epsilon}|_{\epsilon=0}\frac{\partial}{\partial t_{s}}&\hspace{-.1cm}=&\hspace{-.1cm}(\frac{d t_{s}(\epsilon)}{d \epsilon}|_{\epsilon=0})\nnum\\&&\hspace{-.3cm}\times f_{q_{0}}(x^{o}(t^{-}_{s}),u^{o}(t^{-}_{s}))+T\Phi^{(t^{-}_{s},t^{1})}_{f_{q_{0}}}[f_{q_{0}}(x^{o}(t^{1}),u_{1})\nnum\\&&\hspace{-.3cm}-f_{q_{0}}(x^{o}(t^{1}),u^{o}(t^{1}))]\oplus\frac{d t_{s}(\epsilon)}{d \epsilon}|_{\epsilon=0}\frac{\partial}{\partial t_{s}}\in  T_{(x^{o}(t_{s}),t_{s})}\mathcal{S}.\EN
And in case $(ii)$, i.e. $t_{s}(\epsilon)\leq t_{s}$, we have 
\EQ \label{llqq} \frac{d \Phi_{\pi, f_{q_{0}}}^{(t^{-}_{s}(\epsilon),t^{1}),x} }{d\epsilon}|_{\epsilon=0}\oplus\frac{d t_{s}(\epsilon)}{d \epsilon}|_{\epsilon=0}\frac{\partial}{\partial t_{s}}&\hspace{-.1cm}=&\hspace{-.1cm}-\frac{d t_{s}(\epsilon)}{d \epsilon}|_{\epsilon=0}\nnum\\&&\hspace{-.3cm}\times f_{q_{0}}(x^{o}(t^{-}_{s}),u^{o}(t^{-}_{s}))+T\Phi^{(t^{-}_{s},t^{1})}_{f_{q_{0}}}[f_{q_{0}}(x^{o}(t^{1}),u_{1})\nnum\\&&\hspace{-.3cm}-f_{q_{0}}(x^{o}(t^{1}),u^{o}(t^{1}))]\oplus\frac{d t_{s}(\epsilon)}{d \epsilon}|_{\epsilon=0}\frac{\partial}{\partial t_{s}}\in T_{(x^{o}(t_{s}),t_{s})}\mathcal{S}.\EN
Therefore by  (\ref{kir11}) we have
\EQ &&\hspace{-.7cm}\langle dN_{(x^{o}(t^{-}_{s}),t_{s})},\frac{d \Phi_{\pi, f_{q_{0}}}^{(t^{-}_{s}(\epsilon),t^{1}),x} }{d\epsilon}|_{\epsilon=0}\oplus\frac{d t_{s}(\epsilon)}{d \epsilon}|_{\epsilon=0}\frac{\partial}{\partial t_{s}}\rangle=\nnum\\&&\hspace{-.7cm}\langle dN_{x^{o}(t^{-}_{s})}, \frac{d \Phi_{\pi, f_{q_{0}}}^{(t^{-}_{s}(\epsilon),t^{1}),x} }{d\epsilon}|_{\epsilon=0}\rangle+\frac{d t_{s}(\epsilon)}{d \epsilon}|_{\epsilon=0} \langle dN_{t_{s}},\frac{\partial}{\partial t}\rangle=0,\nnum\\\EN
and finally in case $(i)$ we have
\EQ\label{las}\frac{d t_{s}(\epsilon)}{d \epsilon}|_{\epsilon=0}&\hspace{-.1cm}=&\hspace{-.1cm}-\Big(\langle dN_{x^{o}(t^{-}_{s})},f_{q_{0}}(x^{o}(t_{s}),u^{o}(t_{s}))\rangle+\langle dN_{t_{s}},\frac{\partial}{\partial t}\rangle\Big)^{-1}\nnum\\&&\hspace{-.1cm}\times\Big\langle dN_{x^{o}(t^{-}_{s})},T\Phi^{(t^{-}_{s},t^{1})}_{f_{q_{0}}}[f_{q_{0}}(x^{o}(t^{1}),u_{1})-f_{q_{0}}(x^{o}(t^{1}),u^{o}(t^{1}))]\Big\rangle,\nnum\\\EN
and in case $(ii)$
\EQ\label{las1}\frac{d t_{s}(\epsilon)}{d \epsilon}|_{\epsilon=0}&\hspace{-.1cm}=&\hspace{-.1cm}\Big(\langle dN_{x^{o}(t^{-}_{s})},f_{q_{0}}(x^{o}(t_{s}),u^{o}(t_{s}))\rangle+\langle dN_{t_{s}},\frac{\partial}{\partial t}\rangle\Big)^{-1}\nnum\\&&\hspace{-.1cm}\times\Big\langle dN_{x^{o}(t^{-}_{s})},T\Phi^{(t^{-}_{s},t^{1})}_{f_{q_{0}}}[f_{q_{0}}(x^{o}(t^{1}),u_{1})-f_{q_{0}}(x^{o}(t^{1}),u^{o}(t^{1}))]\Big\rangle,\nnum\\\EN
where $\mu$ in (C.14) is given by
\EQ\label{mu2}&& \hspace{-.8cm}\mu=\Big\langle dh(x^{o}(t_{f})),T\Phi^{(t_{f},t_{s})}_{f_{q_{1}}}[f_{q_{1}}(x^{o}(t_{s}),u^{o}(t_{s}))-T\zeta\big(f_{q_{0}}(x^{o}(t^{-}_{s}),u^{o}(t^{-}_{s}))\big)]\Big\rangle \nnum\\&&\hspace{-.5cm}\times\Big( \langle dN_{x^{o}(t^{-}_{s})},f_{q_{0}}(x^{o}(t^{-}_{s}),u^{o}(t^{-}_{s}))\rangle+\langle dN_{t_{s}},\frac{\partial}{\partial t}\rangle\Big)^{-1}.\nnum\\ \EN
Following the steps of the proof of Theorem 4.5 we have
\EQ &&\hspace{-.5cm}p^{o}(t^{-}_{s})=T^{*}\zeta(p^{o}(t_{s}))+\mu dN_{x^{o}(t^{-}_{s})},\nnum\\&& \hspace{-.5cm} p^{o}(t^{-}_{s})\in T^{*}_{x^{o}(t^{-}_{s})}\mathcal{M},\quad p^{o}(t_{s})\in T^{*}_{x^{o}(t_{s})}\mathcal{M},\nnum\\&&\hspace{-.5cm} x^{o}(t_{s})=\zeta(x^{o}(t^{-}_{s})),\EN
where 
\EQ \label{lamlamlam}p^{o}(t):=T^{*}\Phi^{(t^{-}_{s},t)}_{f_{q_{0}}}\circ T^{*}\zeta\circ T^{*}\Phi^{(t_{f},t_{s})}_{f_{q_{1}}}dh(x^{o}(t_{f}))\nnum\\+\mu T^{*}\Phi^{(t^{-}_{s},t)}_{f_{q_{0}}} dv(x^{o}(t^{-}_{s}),t_{s}),\quad t\in [t_{0},t_{s}),\EN
and
\EQ p^{o}(t):=T^{*}\Phi^{(t_{f},t)}_{f_{q_{1}}}dh(x^{o}(t_{f})),\quad t\in [t_{s},t_{f}].\EN

Step 3 in the proof of Theorem 4.5 also holds for time varying switching cases. 
To analyze the possible discontinuity of the Hamiltonian we employ the same method as that used in step 4 of the proof of Theorem 4.5. Therefore
\EQ \label{lam3}&&\hspace{-.6cm}\langle T^{*}\Phi^{(t_{f},t_{s})}_{f_{q_{1}}}dh(x^{o}(t_{f})), [f_{q_{1}}(x^{o}(t_{s}),u^{o}(t_{s}))-T\zeta(f_{q_{0}}(x^{o}(t^{-}_{s}),u^{o}(t^{-}_{s})))]\rangle =\nnum\\&&\hspace{-.6cm}\Big\langle T^{*}\Phi^{(t_{f},t_{s})}_{f_{q_{1}}}dh(x^{o}(t_{f})),\big[ \langle dN_{x^{o}(t^{-}_{s})},f_{q_{0}}(x^{o}(t^{-}_{s}),u^{o}(t^{-}_{s}))\rangle\nnum\\&&\hspace{-.6cm}+\langle dN_{t_{s}},\frac{\partial}{\partial t}\rangle\big]^{-1}\times\Big( \langle dN_{x^{o}(t^{-}_{s})},f_{q_{0}}(x^{o}(t^{-}_{s}),u^{o}(t^{-}_{s}))\rangle+\langle dN_{t_{s}},\frac{\partial}{\partial t}\rangle\Big)\nnum\\&&\hspace{-.6cm}\times[f_{q_{1}}(x^{o}(t_{s}),u^{o}(t_{s}))-T\zeta (f_{q_{0}}(x^{o}(t^{-}_{s}),u^{o}(t^{-}_{s})))]\Big\rangle,\nnum\\\EN
which implies
\EQ &&H_{q_{1}}(x^{o}(t_{s}),p^{o}(t_{s}),u^{o}(t_{s}))=\langle T^{*}\Phi^{(t_{f},t_{s})}_{f_{q_{1}}}dh(x^{o}(t_{f})),f_{q_{1}}(x^{o}(t_{s}),u^{o}(t_{s}))\rangle\hspace{.5cm}\mbox{by C.1 in the main paper}\nnum\\&&=\langle T^{*}\Phi^{(t_{f},t_{s})}_{f_{q_{1}}}dh(x^{o}(t_{f})),T\zeta\big(f_{q_{0}}(x^{o}(t^{-}_{s}),u^{o}(t^{-}_{s}))\big)\rangle\nnum\\&&+\Big\langle T^{*}\Phi^{(t_{f},t_{s})}_{f_{q_{1}}}dh(x^{o}(t_{f})),\nnum\\&&\big\{\big(\langle dN_{x^{o}(t^{-}_{s})},f_{q_{0}}(x^{o}(t_{s}),u^{o}(t_{s}))\rangle+\langle dN_{t_{s}},\frac{\partial}{\partial t}\rangle\big)^{-1}\nnum\\&&\times[f_{q_{1}}(x^{o}(t_{s}),u^{o}(t_{s}))-T\zeta\big(f_{q_{0}}(x^{o}(t^{-}_{s}),u^{o}(t^{-}_{s}))\big)]\big\}\Big\rangle\nnum\\&&\times\big(\langle dN_{x^{o}(t^{-}_{s})},f_{q_{0}}(x^{o}(t_{s}),u^{o}(t_{s}))\rangle+\langle dN_{t_{s}},\frac{\partial}{\partial t}\rangle\big)\hspace{.5cm}\mbox{by \ref{lam3}}\nnum\\&&=\langle T^{*}\zeta \circ T^{*}\Phi^{(t_{f},t_{s})}_{f_{q_{1}}}dh(x^{o}(t_{f})),f_{q_{0}}(x^{o}(t^{-}_{s}),u^{o}(t^{-}_{s}))\rangle \nnum\\&&+\mu\langle dN_{x^{o}(t^{-}_{s})},f_{q_{0}}(x^{o}(t_{s}),u^{o}(t_{s}))\rangle+\mu \langle dN_{t_{s}},\frac{\partial}{\partial t}\rangle\hspace{.5cm}\mbox{by \ref{mu2}},\nnum\\\EN
and finally we have\\

\EQ\label{kos1}\langle p^{o}(t_{s})),f_{q_{1}}(x^{o}(t_{s}),u^{o}(t_{s}))\rangle=\langle p^{o}(t^{-}_{s})),f_{q_{0}}(x^{o}(t^{-}_{s}),u^{o}(t^{-}_{s}))\rangle+\mu\langle dN_{t_{s}},\frac{\partial}{\partial t}\rangle,\nnum\\ \EN
or equivalently
\EQ\label{kos2}\hspace{0cm} H_{q_{0}}(x^{o}(t^{-}_{s}),p^{o}(t^{-}_{s}),u^{o}(t^{-}_{s}))=H_{q_{1}}(x^{o}(t_{s}),p^{o}(t_{s}),u^{o}(t_{s}))-\mu \langle dN_{t_{s}},\frac{\partial}{\partial t}\rangle,\nnum\\ \EN
which completes the proof.\end{proof}

\section{Proof of Theorem \ref{tt22}}
\label{s6}
\begin{proof}
The proof closely parallels the proof of Theorem \ref{tt2} where 
\EQ \label{lamlamlam11}p^{o}(t):=T^{*}\Phi^{(t^{-}_{s},t)}_{f_{q_{0}}}\circ T^{*}\zeta\circ T^{*}\Phi^{(t_{f},t_{s})}_{f_{q_{1}}}dh(x^{o}(t_{f}))\nnum\\+\mu T^{*}\Phi^{(t^{-}_{s},t)}_{f_{q_{0}}} dN_{x^{o}(t^{-}_{s})},\quad t\in [t_{0},t_{s}),\EN
and
\EQ p^{o}(t):=T^{*}\Phi^{(t_{f},t)}_{f_{q_{1}}}dh(x^{o}(t_{f})),\quad t\in [t_{s},t_{f}],\EN
where
\EQ\label{mu3}&& \hspace{-.2cm}\mu=\langle dh(x^{o}(t_{f})),T\Phi^{(t_{f},t_{s})}_{f_{q_{1}}}[f_{q_{1}}(x^{o}(t_{s}),u^{o}(t_{s}))-T\zeta(f_{q_{0}}(x^{o}(t^{-}_{s}),u^{o}(t^{-}_{s})))\nnum\\&&\hspace{0cm}-D_{t}\zeta(x^{o}(t_{s}),t_{s})]\rangle \times\big( \langle dN_{x^{o}(t^{-}_{s})},f_{q_{0}}(x^{o}(t^{-}_{s}),u^{o}(t^{-}_{s}))\rangle+\langle dN_{t_{s}},\frac{\partial}{\partial t}\rangle\big)^{-1}.\nnum\\  \EN 
It should be noted that $D_{t}\zeta(x^{o}(t_{s}),t_{s})(\frac{\partial}{\partial t})\in T\mathcal{M}$ and for  simplicity we drop $\frac{\partial}{\partial t}$. To prove the Hamiltonian discontinuity we have
\EQ \label{lam4}&&\hspace{0cm}\Big\langle T^{*}\Phi^{(t_{f},t_{s})}_{f_{q_{1}}}dh(x^{o}(t_{f})), \Big\{f_{q_{1}}(x^{o}(t_{s}),u^{o}(t_{s}))\nnum\\&&-T\zeta(f_{q_{0}}(x^{o}(t^{-}_{s}),u^{o}(t^{-}_{s})))-D_{t}\zeta(x^{o}(t_{s}),t_{s})\Big\}\Big\rangle \nnum\\&&\hspace{0cm}=\Big\langle T^{*}\Phi^{(t_{f},t_{s})}_{f_{q_{1}}}dh(x^{o}(t_{f})),\big( \langle dv(x^{o}(t^{-}_{s}),t_{s}),f_{q_{0}}(x^{o}(t^{-}_{s}),u^{o}(t^{-}_{s}))\rangle\nnum\\&&\hspace{0cm}+\langle dN_{t_{s}},\frac{\partial}{\partial t}\rangle\big)^{-1}\times\big( \langle dN_{x^{o}(t^{-}_{s})},f_{q_{0}}(x^{o}(t^{-}_{s}),u^{o}(t^{-}_{s}))\rangle+\langle dN_{t_{s}},\frac{\partial}{\partial t}\rangle\big)\nnum\\&&\hspace{0cm}\times[f_{q_{1}}(x^{o}(t_{s}),u^{o}(t_{s}))-T\zeta(f_{q_{0}}(x^{o}(t^{-}_{s}),u^{o}(t^{-}_{s})))-D_{t}\zeta(x^{o}(t_{s}),t_{s})]\Big\rangle,\nnum\\\EN
which implies
\EQ &&H_{q_{1}}(x^{o}(t_{s}),p^{o}(t_{s}),u^{o}(t_{s}))=\big\langle T^{*}\Phi^{(t_{f},t_{s})}_{f_{q_{1}}}dh(x^{o}(t_{f})),f_{q_{1}}(x^{o}(t_{s}),u^{o}(t_{s}))\big\rangle\hspace{.5cm}\mbox{by C.1 in the main paper}\nnum\\&&=\Big\langle T^{*}\Phi^{(t_{f},t_{s})}_{f_{q_{1}}}dh(x^{o}(t_{f})),T\zeta\big(f_{q_{0}}(x^{o}(t^{-}_{s}),u^{o}(t^{-}_{s}))\big)\Big\rangle\nnum\\&&+
\langle T^{*}\Phi^{(t_{f},t_{s})}_{f_{q_{1}}}dh(x^{o}(t_{f})),D_{t}\zeta(x^{o}(t_{s}),t_{s})\rangle\nnum\\&&+\Big\langle T^{*}\Phi^{(t_{f},t_{s})}_{f_{q_{1}}}dh(x^{o}(t_{f})),\Big(\langle dN_{x^{o}(t^{-}_{s})},f_{q_{0}}(x^{o}(t_{s}),u^{o}(t_{s}))\rangle+\langle dN_{t_{s}},\frac{\partial}{\partial t}\rangle\Big)^{-1}\nnum\\&&\times[f_{q_{1}}(x^{o}(t_{s}),u^{o}(t_{s}))-T\zeta (f_{q_{0}}(x^{o}(t^{-}_{s}),u^{o}(t^{-}_{s})))]\Big\rangle\nnum\\&&\times\Big(\langle dN_{x^{o}(t^{-}_{s})},f_{q_{0}}(x^{o}(t_{s}),u^{o}(t_{s}))\rangle+\langle dN_{t_{s}},\frac{\partial}{\partial t}\rangle\Big)\hspace{.5cm}\mbox{by \ref{lam4}}\nnum\\&&=\langle T^{*}\zeta \circ T^{*}\Phi^{(t_{f},t_{s})}_{f_{q_{1}}}dh(x^{o}(t_{f})),f_{q_{0}}(x^{o}(t^{-}_{s}),u^{o}(t^{-}_{s}))\rangle \nnum\\&&+\mu\langle dN_{x^{o}(t^{-}_{s})},f_{q_{0}}(x^{o}(t^{-}_{s}),u^{o}(t^{-}_{s}))\rangle \nnum\\&&+D^{*}_{t}\zeta(T^{*}\Phi^{(t_{f},t_{s})}_{f_{q_{1}}}dh(x^{o}(t_{f})))+\mu \langle dN_{t_{s}},\frac{\partial}{\partial t}\rangle\hspace{.5cm}\mbox{by \ref{mu3}}\nnum\\&&=H_{q_{0}}(x^{o}(t^{-}_{s}),p^{o}(t^{-}_{s}),u^{o}(t^{-}_{s}))+D^{*}_{t}\zeta(T^{*}\Phi^{(t_{f},t_{s})}_{f_{q_{1}}}dh(x^{o}(t_{f})))+\mu \langle dN_{t_{s}},\frac{\partial}{\partial t}\rangle,\nnum\\\EN
where by the definition of pullbacks (see \cite{Lee2})  
\EQ\hspace{.5cm}\langle T^{*}\Phi^{(t_{f},t_{s})}_{f_{q_{1}}}dh(x^{o}(t_{f})),D_{t}\zeta(x^{o}(t_{s}),t_{s})\rangle= D^{*}_{t}\zeta(T^{*}\Phi^{(t_{f},t_{s})}_{f_{q_{1}}}dh(x^{o}(t_{f})))\in \mathds{R},\EN
and $p^{o}(t_{s})=T^{*}\Phi^{(t_{f},t_{s})}_{f_{q_{1}}}dh(x^{o}(t_{f}))$. The remaining of the proof is similar to that of  (\ref{kos1}) and (\ref{kos2}).\end{proof}

 \bibliographystyle{siam.bst}
\bibliography{HSCC}
 
\end{document}